\documentclass[11pt, psamsfonts]{amsart}

\usepackage{amssymb,amsfonts}
\usepackage[all,arc]{xy}
\usepackage{enumerate}
\usepackage{mathrsfs}
\usepackage{thmtools}
\usepackage{thm-restate}
\usepackage{datetime}
\usepackage{hyperref}
\usepackage{graphicx}
\usepackage{cleveref}
\usepackage{comment}

\calclayout

\newcommand\blfootnote[1]{%
  \begingroup
  \renewcommand\thefootnote{}\footnote{#1}%
  \addtocounter{footnote}{-1}%
  \endgroup
}

\newtheorem{thm}{Theorem}[section]
\newtheorem{cor}[thm]{Corollary}

\newtheorem{lem}[thm]{Lemma}
\newtheorem{conj}[thm]{Conjecture}

\theoremstyle{definition}

\theoremstyle{remark}

\makeatletter
\let\c@equation\c@thm
\makeatother
\numberwithin{equation}{section}

\makeatletter
\def\subsection{\@startsection{subsection}{3}%
  \z@{.5\linespacing\@plus.7\linespacing}{.1\linespacing}%
  {\bfseries}}
\makeatother

\newcommand{\N} { \mathbb{N}}

\newcommand{\Z} { \mathbb{Z}}

\renewcommand{\th} { \text{th}}

\newcommand{\cut} { \backslash}

\newcommand{\Lon} { \mbox{Lon}}
\newcommand{\Soc} { \mbox{Soc}}
\newcommand{\isom}{\cong}
\newcommand{\Cay} { \mbox{Cay}}

\DeclareMathOperator{\IR}{IR}
\DeclareMathOperator{\ir}{ir}
\DeclareMathOperator{\pn}{pn}
\DeclareMathOperator{\rep}{rep}

\usepackage{etoolbox}
\makeatletter
\patchcmd{\@maketitle}
  {\ifx\@empty\@dedicatory}
  {\ifx\@empty\@date \else {\vskip3ex \centering\footnotesize\@date\par\vskip1ex}\fi
   \ifx\@empty\@dedicatory}
  {}{}
\patchcmd{\@adminfootnotes}
  {\ifx\@empty\@date\else \@footnotetext{\@setdate}\fi}
  {}{}{}
\makeatother

\title{Domination Parameters of the Unitary Cayley Graph of $\mathbb{Z}/n\mathbb{Z}$}
\author{Amanda Burcroff}

\date{ \vspace{-0.3in}{\it University of Michigan\\ 500 S. State St.\\  Ann Arbor, MI 48109 United States}}

\begin{document}
\maketitle
\begin{abstract}
    \vspace{-0.5in}
    The unitary Cayley graph of $\mathbb{Z}/n\mathbb{Z}$, denoted $X_n$, is the graph on  $\{0,\dots,n-1\}$ where vertices $a$ and $b$ are adjacent if and only if $\gcd(a-b,n) = 1$.   We answer a question of Defant and Iyer by constructing a family of infinitely many integers $n$ such that $\gamma_t(X_n) \leq g(n) - 2$, where $\gamma_t$ denotes the total domination number and $g$ denotes the Jacobsthal function.  We determine the irredundance number, domination number, and lower independence number of certain direct products of complete graphs and give bounds for these parameters for any direct product of complete graphs.  We provide upper bounds on the size of irredundant sets in direct products of balanced, complete multipartite graphs which are asymptotically correct for the unitary Cayley graphs of integers with a bounded smallest prime factor.
\end{abstract}

\blfootnote{{\it E-mail address:} burcroff@umich.edu}
\vspace{-0.5in}
\section{Introduction}
For a group $\Gamma$ and a set $S = S^{-1} \subseteq \Gamma$ not containing the identity element, the {\it Cayley graph} $\Cay(\Gamma;S)$ is the undirected graph with vertices labeled by $\Gamma$ and edge set $\{\{a,b\} : a -b \in S\}$.  If $\Gamma$ is a commmutative ring with unity, the {\it unitary Cayley graph} $X_\Gamma$ is the Cayley graph $\Cay(\Gamma,U_\Gamma)$, where $U_\Gamma$ is the set of units in $\Gamma$.  More information on Cayley graphs and unitary Cayley graphs can be found in the algebraic graph theory texts by Biggs \cite{Big} and by Godsil and Royle \cite{GR}.  We are interested in $X_n = X_{\Z/n\Z}$, that is, the graph on $\{0,\dots,n-1\}$ where vertices $a$ and $b$ are connected by an edge if and only if $\gcd(a-b,n) = 1$.   Observe that $X_n$ is vertex-transitive and regular of degree $\phi(n)$, where $\phi$ denotes the Euler totient function.  

These graphs were cast into the limelight in 1989, when Erd{\H o}s and Evans showed that every finite simple graph $G$ is isomorphic to an induced subgraph of $X_n$ for some positive integer $n$, in which case they say $G$ is {\it representable modulo $n$}.  The {\it representation number} $\rep(G)$ of a graph $G$ is the minimum positive integer such that $G$ is representable modulo $n$.  The representation numbers of many classes of graphs have been determined \cite{Akh, AEP, AEP2, Eva, EIN, NU}.
See Section 7.6 of \cite{Gal} for a survey of representation numbers and additional references.  

In this paper, we consider domination parameters of the unitary Cayley graph of $\Z/n\Z$, including variants of the domination, irredundance, and independence numbers.  Recall that a set $S \subseteq V(G)$ is called {\it independent} if it contains no pair of adjacent vertices.  A set $S \subseteq V(G)$ is called {\it dominating} if every vertex of $G$ is either contained in $S$ or adjacent to a vertex of $S$.  A set $S \subseteq V(G)$ is called {\it irredundant} if for each $v \in S$, either $v$ is isolated in $S$ or $v$ has a neighbor $u\not\in S$ such that $u$ is not adjacent to any vertex of $S \cut \{v\}$.   These notions yield the following graph parameters for a graph $G$.

\begin{itemize}
\item The {\it irredundance number} $\ir(G)$ is the minimum size of a maximal irredundant set.
\item The {\it domination number} $\gamma(G)$ is the minimum size of a dominating set.
\item The {\it lower independence number} $i(G)$, also known as the {\it independent domination number}, is the minimum size of a maximal independent set.
\item The {\it independence number} $\alpha(G)$ is the maximum size of an independent set.
\item The {\it upper domination number} $\Gamma(G)$ is the maximum size of a minimal dominating set.
\item The {\it upper irredundance number} $\IR(G)$ is the maximum size of an irredundant set.
\end{itemize}

For any graph $G$, we have the following chain of inequalities, known as the {\it domination chain}:
$$\ir(G) \leq \gamma(G) \leq i(G) \leq \alpha(G) \leq \Gamma(G) \leq \IR(G).$$
Hundreds of papers have been written showing that some inequalities in this chain are equalities for certain classes of graphs, see Section 3.5 of \cite{HHS} for more details.  Many of these results were unified in 1994, when Cheston and Fricke showed that $\alpha(G) = \IR(G)$ for any strongly perfect graph $G$ \cite{CF}.

It has been shown in \cite{ABJJKKP} that the unitary Cayley graph of any finite commutative ring is a direct product of balanced, complete multipartite graphs, so we will often work in this more general setting.  The {\it direct product} (sometimes called the {\it tensor product} or {\it Kronecker product}) of two graphs $G$ and $H$, denoted $G \times H$ (alternatively, $G \otimes H$), has vertex set $V(G) \times V(H)$ with $(g_1,h_1)$ adjacent to $(g_2,h_2)$ if and only if $g_1$ is adjacent to $g_2$ in $G$ and $h_1$ is adjacent to $h_2$ in $H$.  Throughout this paper, $\prod_{i = 1}^t G_i$ will denote the direct product of the graphs $G_1,\dots,G_t$.  The {\it balanced, complete $b$-partite graph with partite set size $a$}, denoted $K[a,b]$, is the graph on $ab$ vertices partitioned into $b$ partite sets of size $a$ such that two vertices are adjacent if and only if they lie in different partite sets. If $p$ is a prime and $\alpha$ is a positive integer, note that $X_{p^\alpha} \isom K[p^{\alpha-1},p]$.  An application of the Chinese remainder theorem shows that for an integer $n$ with prime factorization $n = p_1^{\alpha_1}p_2^{\alpha_2} \cdots p_t^{\alpha_t}$, we have $X_n \isom \prod_{i = 1}^t X_{p_i^{\alpha_i}} \isom \prod_{i = 1}^t K[p_i^{\alpha_i - 1}, p_i]$.

In this paper, we build upon the work of Defant and Iyer \cite{DI} to determine domination parameters of the unitary Cayley graphs of  $\Z/n\Z$.  Let $g(n)$ denote the minimum positive integer $m$ such that every set of $m$ consecutive integers contains an integer which is coprime to $n$; this arithmetic function is known as the {\it Jacobsthal function}.  The {\it total domination number} of a graph $G$ is the minimum size of a set $S$ in $G$ such that every vertex is adjacent to a member of $S$.  Defant observed in \cite{Def} that there exist integers $n$ such that $\gamma_t(X_n) \leq  g(n) - 1$.  In Section \ref{Dom Sets}, we answer two questions of Defant and Iyer in the positive.  The first asks whether there exists a single integer $n$ such that $\gamma_t(X_n) \leq g(n) - 2$, and the second asks whether there exist integers $n$ with arbitrarily many distinct prime factors such that $\gamma(X_n) \leq g(n) - 2$.  We construct integers $n$ with arbitrarily many distinct prime factors such that $X_n$ contains a dominating cycle of size $g(n)-2$; this answers both questions of Defant and Iyer since a dominating cycle is necessarily a total dominating set. 

In Section \ref{Low Dom}, we provide bounds on the irredundance, domination, and lower independence numbers of direct products of complete graphs and determine these parameters in certain cases. One application of this work is the construction of some infinite families of integers $n$ where $\ir(X_n) = \gamma(X_n) = i(X_n)$. We provide an upper bound on the lower independence number of $X_n$ which disproves a claim of Uma Maheswari and Maheswari \cite{MahMah}.  Defant and Iyer \cite {DI} recently determined the value of $\gamma(\prod_{i = 1}^4 K_{n_i})$ in several cases; we compute this parameter in all cases. Lastly, in Section \ref{Upp Dom}, we provide upper bounds on the sizes of irredundant sets in direct products of balanced, complete multipartite graphs.  In the case of unitary Cayley graphs of $\Z/n\Z$, Theorem \ref{up dom bd 1} yields the following bound.

\setcounter{thm}{4}
\setcounter{section}{5}
\begin{cor}
Let $n = p_1^{\alpha_1}p_2^{\alpha_2}\cdots p_t^{\alpha_t}$, where $p_1 < \cdots < p_t$.  Then 
$$\IR(X_n) \leq \left(1 + 2 \cdot \frac{p_1}{p_t}\cdot \frac{1}{p_1^{\alpha_1-1}p_2^{\alpha_2-1}\cdots p_t^{\alpha_t - 1}}\right)\alpha(X_n).$$
\end{cor}
\setcounter{thm}{1}
\setcounter{section}{1}

\section{Preliminaries}
A graph $G$ is a set of vertices $V(G)$ along with a set of undirected edges $E(G)$, excluding loops.  For any $U \subseteq V(G)$, the subgraph of $G$ induced by $U$, denoted $G[U]$, is the graph with vertex set $U$ and whose edge set is precisely the edge set $E(G)$ restricted to $U \times U$.  

Let $N(v)$ denote the neighborhood of a vertex $v$, the set of all vertices adjacent to $v$.  Let $N[v]$ denote the {\it closed neighborhood} of a vertex $v$, that is, the set $N(v)$ along with the vertex $v$ itself.  For $S \subseteq V(G)$, let $N[S] = \bigcup_{v \in S} N[v]$.  A vertex $u \in V(G)$ is a {\it private neighbor} of a vertex $v \in S \subseteq V(G)$ if $u \in (N[v] \cut N[S \cut \{v\}])$.  Note $u$ can equal $v$.  Let $\pn[v] = \pn[v;S]$ denote the set of private neighbors of $v \in S$.  Let $\pn[S] = \bigcup_{v \in S} \pn[v;S]$.  Note that a set $S \subseteq V(G)$ is irredundant if and only if every $v \in S$ has a private neighbor, and $S \subseteq V(G)$ is dominating if and only if $N[S] = V(G)$.  

Let $\N = \{1,2,3,\dots\}$. For $n \in \N$, let $\omega(n)$ denote the number of distinct prime factors of $n$.  For $S \subseteq \N$, let $\omega(S) = \{\omega(n) : n \in S\}$.

\section{Dominating Cycles in the Unitary Cayley Graphs of $\Z/n\Z$}\label{Dom Sets}

It was shown by Maheswari and Manjuri  \cite{MahMan} that the value of the Jacobsthal function $g(n)$ is an upper bound for the domination number of $X_n$, the unitary Cayley graph of $\Z/n\Z$.  Defant and Iyer \cite{DI} note that the stronger inequality $\gamma_t(X_n) \leq g(n)$ holds and that these quantities can differ by $1$ for $n$ with arbitrarily many distinct prime factors.

We consider a variation of the domination number, introduced by Veldman \cite{Vel} in 1983.  The {\it cycle domination number} of a graph $G$, denoted $\gamma_c(G)$, is the minimum size of a dominating cycle in $G$, provided that such a cycle exists.  Note that $\gamma_c(G) \geq \gamma_t(G) \geq \gamma(G)$ for any graph $G$.  Since $(0,1,\dots,n-1)$ is a cycle in $X_n$, the cycle domination number of the unitary Cayley graph of $\Z/n\Z$ exists for all $n \in \N$.

In the following theorem, we exhibit an infinite family of integers $n$ such that $g(n) - \gamma_t(X_n) \geq g(n) - \gamma_c(X_n) \geq 2$. Let $M_{c,j}$  be the set of positive integers $n$ such that $g(n) - \gamma_c(X_{\Z/n\Z}) \geq j$. Similarly, let $M_{t,j}$  be the set of positive integers $n$ such that $g(n) - \gamma_t(X_{\Z/n\Z}) \geq j$. 

\begin{thm}\label{cycle dom}
The set $\omega(M_{c,2})$ is unbounded.
\end{thm}
\begin{proof}
\vspace{-0.5cm}
Let $q$ be a prime such that $q \equiv 1 \mod 3$.  Let $k = \frac{2q - 2}{3}$, and let $p_1,\dots,p_k$ be primes such that $2q + 10 < p_1 < \dots < p_k$.  Let $n = 6q\prod_{j=1}^k p_j$.  

We begin by showing $g(n) \geq 2q + 8$.  For $i \in \{0,\dots,2q + 6\}$, let 
\[a_i =\begin{cases} 
2 & \text{if } i \equiv 0 \mod 2\\
3 & \text{if } i \equiv 1 \mod 6\\
q & \text{if } i = 3 \text{ or } i = 2q + 3\\
\min(\{p_1,\dots,p_k\} \cut \{a_0,\dots,a_{i-1}\}) & \text{otherwise.}
\end{cases}
\]

Let $L(x) = \{i \in \{0,\dots,2q + 6\} : a_i = x\}$. Note that $L(2)| = q + 4$, $|L(3)| = \frac{q + 5}{3}$, and $|L(q)| = 2$, and $|L(p_j)| = 1$ for each $j \in \{1,\dots,k\}$.  By the Chinese remainder theorem, there exists a $z \in \Z$ such that $z \equiv -i \mod a_i$ for all $i$.  The set $\{z + i : 0 \leq i \leq 2q + 6\}$ is a set of $2q + 7$ integers, none of which are relatively prime to $n$.  Hence $g(n) \geq 2q + 8$.

Let $y$ be the unique vertex of $X_{\Z/n\Z}$ such that $y \equiv 0 \mod 2$, $y \equiv 2 \mod 3$, $y \equiv -1 \mod q$, and $y \equiv -1 \mod p_i$ for all $1 \leq i \leq k$. Let $z$ be the unique vertex of $X_{\Z/n\Z}$ such that $z \equiv 1 \mod 2$, $z \equiv 0 \mod 3$, $z \equiv -2 \mod q$, and $z \equiv -2 \mod p_i$ for all $1 \leq i \leq k$. Let $D = \{0,1,\dots,2q + 3,y,z\}$.  We will show that the vertices of $D$ form a cycle dominating set of $X_{\Z/n\Z}$.  Since $|D| = 2q + 6 \leq  g(n)-2$, this will prove $n \in M_{c,2}$.  

Suppose a vertex $x$ is not adjacent to any element of $D \cut \{y,z\}$.  We will show that $x$ is adjacent to either $y$ or $z$.  The set $S = \{x,x-1,x-2,\dots,x-(2q + 3)\}$ consists of $2q + 4$ consecutive integers, none of which are coprime to $n$.  For $r \in \N$, let $B(r) = \{s \in S : s \equiv 0 \mod r\}$.  Let $B(2,3) = \{s \in S : s \equiv 0 \mod 2 \text{ or } s \equiv 0 \mod 3\}$.  Observe that $|B(2,3)| = \frac{2}{3}(2q + 4) = \frac{4q + 8}{3}$, $|B(q)| \leq 3$, and $|B(p_i)|\leq 1$ for all $p_i$.  Since 
$$S = B(2,3) \cup B(q) \cup \bigcup_{i = 1}^{k} B(p_i),$$
we have
\begin{align*}
    2 q + 4 &= \left|B(2,3) \cup B(q) \cup \bigcup_{i = 1}^{k} B(p_i)\right| \leq |B(2,3)| + |B(q)| + \sum_{i = 1}^k |B(p_i)|\\
    &\leq \frac{4q + 8}{3} + 3 + \frac{2q-2}{3} = 2q + 5.
\end{align*}

This calculation implies that $|B(q) \cut B(2,3)| \geq 2$. Let $\ell$ denote the minimum nonnegative integer such that $x - \ell \in B(q)$.  If $\ell > 3$, then one element of $\{x - \ell, x - \ell - q\} = B(q)$ is even, hence $|B(q) \cut B(2,3)| < 2$.  Therefore,  we must have $\ell \leq 3$ with $B(q) \cap B(2,3) = \{x - \ell - q\}$.  Hence the sets $B(2,3)$, $B(q) \cut \{x - \ell - q\}$, and $\bigcup_{i = 1}^{k} B(p_i)$ are disjoint.
 
Thus, exactly one of $x,x-1,x-2,x-3$ is contained in $B(q)$.  In particular, $x \not\equiv -1,-2 \mod q$.  For each $1 \leq i \leq k$, from the fact that $|B(p_i)| = 1$ and the assumption that $p_i \geq 2q + 10$, we can conclude that $x \not\equiv -1,-2 \mod p_i$.

If $x$ is not adjacent to $y$, then either $x \equiv 2 \mod 3$ or $x \equiv 0 \mod 2$.  We will show that, under these conditions, $x \equiv 2 \mod 6$ or $x \equiv 4 \mod 6$.  Since $z \equiv 3 \mod 6$, this is enough to show that $x$ is adjacent to $z$.

Suppose $x \equiv 2 \mod 3$.  
\begin{itemize}
    \item If $x \equiv 0 \mod 2$, then $x \equiv 2 \mod 6$.
    \item If $x \equiv 1 \mod 2$, then $x \equiv 5 \mod 6$.  Thus $x-1,x-2,x-3 \in B(2,3)$, so $x \in B(q)$.  Hence $x - 2q \in B(q)$.  However, $x - 2q \equiv 2 - 2 \equiv 0 \mod 3$, contradicting the disjointedness of $B(2,3)$ and $B(q) \cut \{x - \ell - q\}$.
\end{itemize}
Now suppose $x \not\equiv 2 \mod 3$ and $x \equiv 0 \mod 2$.
\begin{itemize}
    \item If $x \equiv 1 \mod 3$, then $x \equiv 4 \mod 6$. 
    \item If $x \equiv 0 \mod 3$, then $x \equiv 0 \mod 6$.  Thus $x,x-2, x-3 \in B(2,3)$.  Hence $x - 1 \in B(q)$.  However, $(x - 1) - 2q \equiv 2 - 2 \equiv 0 \mod 3$, contradicting the disjointedness of $B(2,3)$ and $B(q) \cut \{x - \ell - q\}$.
\end{itemize}
Thus, if $x$ is not adjacent to $y$, then $x \equiv 4 \mod 6$ or $x \equiv 2 \mod 6$.  We can conclude that $D$ is indeed a total dominating set, and we have 
$$\gamma_t(X_{\Z/n\Z}) \leq 2q + 6 = (2q + 8) - 2 \leq g(n) - 2.$$

Lastly, note that $y$ is adjacent to $1$, $z$ is adjacent to $2$, and $y$ is adjacent to $z$.  Therefore, $(0,1,y,z,2,3,\dots,2q + 3)$ is a dominating cycle.
\end{proof}

We will briefly expand upon the motivation behind the construction in Theorem \ref{cycle dom}.  Fix an integer $d$ and a prime $q$.  Let $[n]_d$ denote the smallest nonnegative integer equivalent to $n$ modulo $d$.  Let $R_{q,d}$ be the set of integers $x \in \{0,\dots,d-1\}$ such that $x$ and $[x-2q]_d$ are relatively prime to $d$.  Let $R_{q,d,k} = \{x + \ell : x\in R_{q,d} \text{ and }0 \leq \ell \leq k-1\}$. A key property used in both the construction of integers $n$ for which $\gamma_t(X_n) \leq g(n) - 1$ by Defant and Iyer \cite{DI} (using $d = 2$ and $k = 2$) and the construction in Theorem \ref{cycle dom} (using $d = 6$ and $k = 4$) is that $R_{q,d,k}$ can be covered by relatively few vertices in $X_d$, namely $1$ vertex in \cite{DI} and $2$ in Theorem \ref{cycle dom}.  Note that $k$ is the minimum integer such that, for $x$ not adjacent the consecutive vertices of the constructed dominating set, one of $x,x-1,\dots,x-(k-1)$ is divisible by $q$.  If other such triples of integers $(q,d,k)$ can be found, similar constructions could yield other families of integers in $M_{c,2}$ and perhaps even $M_{c,j}$ for $j \geq 3$.

Budadoddi and Mallikarjuna Reddy claim in \cite{BR} that the cycle dominating number (see Section \ref{Dom Sets}) of $X_n$ is given by the Jacobsthal function $g(n)$, provided $n$ is neither a prime power nor twice a prime power.  Theorem \ref{cycle dom} shows that this is not the case; in fact, there are integers $n$ with $\omega(n)$ arbitrarily large such that $g(n) - \gamma_c(X_n) \geq 2$.  

In the family constructed in Theorem \ref{cycle dom}, we see that $\gamma_c(X_n) \leq g(n)$.  However, this inequality does not hold for all integers.  For example, it is easily seen that $\gamma_c(X_6) = 6$ while $g(6) = 4$.  We do not know if there exist infinitely many integers for which $\gamma_c(X_n) > g(n)$.  The construction of a dominating set in $X_n$ of size $g(n)$ by Manjuri and Maheswari \cite{MahMan} shows that $\gamma_c(X_n) \leq g(n)$ whenever $\gcd(n,g(n)) = 1$.

Theorem \ref{cycle dom} also answers two questions of Defant and Iyer from \cite{DI}, the first asking whether $M_{t,2}$ is nonempty and the second asking whether there exist integers $n$ with $\omega(n)$ arbitrarily large such that $\gamma(X_n) \leq g(n)-2$.  

\begin{cor}
There exist integers $n$ with arbitrarily many distinct prime factors such that $\gamma(X_n) \leq \gamma_t(X_n) \leq g(n) - 2$.
\end{cor}
This leads to the natural next question: does there exist a single integer such that $\gamma_t(X_n) \leq g(n) - 3$?

\section{Lower Domination Parameters in Products of Complete Graphs}\label{Low Dom}
In this section, we consider the quantities in the lower portion of the domination chain for products of complete graphs.  It is often useful to think of vertices in $\prod_{i = 1}^t K_{n_i}$ as $t$-tuples of integers where the $i^{\th}$-entry is in the range $\{0,\dots,n_i-1\}$, where two vertices are adjacent if and only if their corresponding vectors differ in every coordinate.  For squarefree positive integers $n$, we refer to vertices in $X_n$ as integers and tuples interchangeably.

\subsection{Irredundant Sets in Products of Complete Graphs}
We will make use of two previous results; the first from Defant and Iyer in \cite{DI}, and the second from Bollob{\'a}s and Cockayne, as well as Allan and Laskar, independently, in \cite{BC, AL}.
\begin{thm}\emph{(\cite{DI})}
Let $G = \prod_{i = 1}^t K_{n_i}$, where $2 \leq n_1 \leq n_2 \leq \cdots \leq n_t$, $t \geq 4$, and $n_2 \geq 3$.  We have
$$\gamma(G) \geq t + 1 + \left\lfloor \frac{t-1}{n_1 - 1} \right\rfloor$$ 
\end{thm}
\begin{thm}\emph{(\cite{BC,AL})}
For any graph $G$, $\ir(G) \geq \frac{1}{2}(\gamma(G) + 1)$.
\end{thm}
\begin{thm}\label{dom numbers}
Let $G = \prod_{i = 1}^t K_{n_i}$, where $2 \leq n_1 \leq n_2 \leq \cdots \leq n_t$.  If $t = 1$, then $\ir(G) = 1$.  If $t = 2$, then
$$\ir(G) = \begin{cases} 2, & \text{if } n_1 = 2;\\ 3, & \text{if } n_1 \geq 3.\end{cases}$$
If $t = 3$, then $\ir(G) = 4$.  For $t \geq 4$, we have $\ir(G) \geq \frac{1}{2}\left( t + \left\lfloor \frac{t-1}{n_1 - 1}\right\rfloor\right) + 1$.
\end{thm}
\begin{proof}
\vspace{-0.1cm}
As $\ir(G) \leq \gamma(G)$, the calculation of the domination number for $G = \prod_{i = 1}^t K_{n_i}$ for $t \leq 3$ by Meki{\v s} \cite{Mek} proves that these irredundance numbers are at most the stated values. The case $t=1$ is trivial, since any single vertex is a maximal irredundant set.  

Let $t = 2$ and $n_1 = 2$.  Fix an irredundant set $\{(x_1,y_1)\}$.  We claim this set is not a maximal irredundant set.  This follows from the fact that $\{(x_1,y_1),(1-x_1,y_1)\}$ is also irredundant, as each vertex is its own private neighbor.  Therefore, $\ir(G) \geq 2$.

Let $t = 2$, and suppose $n_1 \geq 3$. Clearly no vertex of $G$ is dominating, so $\ir(G) > 1$.  Fix an irredundant set $S = \{(x_1,y_1),(x_2,y_2)\}$.  Suppose these two vertices are equal in some coordinate, without loss of generality $x_1 = x_2$.  Fix $y_3 \in \{1,2,3\}\cut\{y_1,y_2\}$.  The set $\{(x_1,y_1),(x_2,y_2), (x_1,y_3)\}$ is an independent hence irredundant set.  Thus, we can assume $x_1 \not= x_2$ and $y_1 \not= y_2$.  Let $S' = \{(x_1,y_1),(x_2,y_2),(x_1,y_2)\}$.  Fix $z_1 \in \{1,2,3\} \cut \{x_1,x_2\}$ and $z_2 \in \{1,2,3\}\cut\{y_1,y_2\}$.  Then $(z_1,y_2)$ is a private neighbor of $(x_1,y_1)$ in $S'$, $(x_1,z_2)$ is a private neighbor of $(x_2,y_2)$ in $S'$, and $(x_1,y_2)$ is its own private neighbor in $S'$.  Therefore, the minimum size of a maximal irredundant set in $G$ is at least $3$. 

Let $t = 3$.  Suppose, seeking a contradiction, that 
$$S = \{(x_1,y_1,z_1), (x_2,y_2,z_2), (x_3,y_3,z_3)\}$$ 
is a maximal irredundant set in $G$.  We can assume these three vertices are not all equal in any coordinate, otherwise $S$ can be extended to an independent set of size $4$ by taking a fourth vector which is also equal in that coordinate.  

Suppose that $S$ is independent; without loss of generality assume $x_1 = x_2 \not= x_3$, $y_1 = y_3 \not= y_2$, and $z_2 = z_3 \not= z_1$.  Note that if $n_3 = 2$, every irredundant set is independent.  The point $(x_3,y_2,z_1)$ is not in $S$, and we have that $S \cup \{(x_3,y_2,z_1)\}$ is independent.  This contradicts that $S$ is a maximal irredundant set.  

Thus, we can assume $S$ is not independent; without loss of generality assume 
$$S = \{(0,0,0),(1,1,1),(x_3,y_3,z_3)\}.$$
If $(x_3,y_3,z_3) = (2,2,2)$, then $S \cup \{(0,1,2)\}$ is irredundant, contradicting the maximality of $S$.  Thus, we can assume $x_3 = 0$.

Since $S$ is a maximal irredundant set, we cannot add $(0,0,1)$ or $(0,1,0)$ to $S$ without removing the irredundance property.  Since the set $\{(0,0,0),(1,1,1),(0,0,1),(0,1,0)\}$ is irredundant, we have that $(x_3,y_3,z_3) \not= (0,0,1), (0,1,0)$.  Suppose, seeking a contradiction, that we cannot add either $(0,0,1)$ or $(0,1,0)$.  
Since neither $(0,0,1)$ nor $(0,1,0)$ is contained in $N[S]$, each of $N[(0,0,1)]$ and $N[(0,1,0)]$ must contain at least one of $\pn[(0,0,0)]$, $\pn[(1,1,1)]$, or $\pn[(0,y_3,z_3)]$.  

If $N[(0,0,1)] \supseteq \pn[(0,0,0)]$, we must have  $y_3 = 0$ or $1$, lest $(1,y_3,1) \in \pn[(0,0,0)] \cut N[(0,0,1)]$.  If $y_3 = 1$, then $z_3 = 1$ since we have shown $(0,1,0)\not\in S$.  However, the set $S \cup \{(1,0,1)\}$ is irredundant.  Hence $y_3 = 0$.  This also leads to a contradiction, as $(0,0,0) \not= (0,y_3,z_3)$ and $(0,0,1)\not\in S$.   If $N[(0,0,1)] \supseteq \pn[(1,1,1)]$, then we also reach a contradiction, as $x_3 \geq 3$ implies $(0,0,2) \in \pn[(1,1,1)] \cut N[(0,0,1)]$.  Therefore $N[(0,0,1)] \supseteq \pn[(0,y_3,z_3)]$.  

By permuting coordinates, we can similarly show that $N[(0,1,0)] \supseteq \pn[(0,y_3,z_3)]$.  Hence $N[(0,0,1)]\cap N[(0,1,0)] \supseteq \pn[(0,y_3,z_3)]$.  It is straightforward to check that under this condition $y_3,z_3\not\in \{0,1\}$.  Without loss of generality we can assume $(0,y_3,z_3) = (0,2,2)$.  However, we reach our final contradiction from the fact that $S$ can be extended to the irredundant set $\{(0,0,0),(1,1,1),(0,2,2),(1,1,0)\}$.
\end{proof}

\subsection{Dominating Sets in Products of Complete Graphs}
We provide an upper bound on the domination number of any product of $t$ complete graphs which is exponential in $t$.  This is  an improvement on the upper bound yielded by a theorem of Bre{\v s}ar, Klav{\v z}ar, and Rall \cite{BKR}, stating that $\gamma(G \times H) \leq 3 \gamma(G)\gamma(H)$ for any graphs $G$ and $H$.  This implies $\gamma\left(\prod_{i=1}^t K_{n_i}\right) \leq 3^{t-1}$.  We show that $\gamma\left(\prod_{i=1}^t K_{n_i}\right) \leq 3 \cdot 2^{t-2}$.

\begin{thm}\label{dom code}
Let $G = \prod_{i = 1}^t K_{n_i}$.  Let $M$ be a family of vertices in $\{0,1\}^t \subseteq V(G)$ such that no two vertices in $M$ are equal in $t-1$ coordinates or different in all $t$ coordinates.  Then $\{0,1\}^t \cut M$ is a dominating set for $G$.
\end{thm}
\begin{proof}
Let $D = \{0,1\}^t \cut M$.  Suppose $v \in \{0,1\}^t$.  By the requirement that no two vertices in $M$ differ in all coordinates, at least one of $v$ and its Boolean complement $(1,\dots,1) - v$ is in $D$.  Hence $v$ is dominated by $D$.  

Suppose $u \in V(G) \cut \{0,1\}^t$.  Let $\ell \in \{1,\dots,t\}$ be a coordinate in which $u$ is neither $0$ nor $1$.  Observe that there exist two vertices $w_1$ and $w_2$ in $\{0,1\}^t$, differing in only the $\ell^{\text{th}}$ coordinate, which differ from $u$ in every coordinate.  This implies $w_1$ and $w_2$ are both adjacent to $u$ in $G$.  Since $w_1$ and $w_2$ are equal in $t-1$ coordinates, at least one of them is in $D$.  Therefore, $u$ is dominated by $D$.  We conclude that $D$ is a dominating set of $G$.
\end{proof}

Let $A(t,d,t-1)$ denote the maximum number of binary vectors of length $t$ such that any two distinct vectors have Hamming distance between $d$ and $t-1$, inclusive.  Let $A(t,d)$ denote the maximum number of binary vectors such that no two vectors have Hamming distance less than $d$.  By taking a set witnessing $A(t,d)$ and throwing out one of the vectors in any pair of Boolean complements, we obtain the following bound.

A classic result in coding theory is the Gilbert-Varshamov lower bound \cite{Gil,Var} on $A(t,d)$, originally stated for alphabets of prime power size.

\begin{thm}\emph{(\cite{Gil,Var})}\label{GV bound}
If $k$ satisfies
$$2^k < \frac{2^t}{\sum_{j = 0}^{d - 2} {t \choose j}},$$
then $A(t,d) \geq 2^k$.
\end{thm}

\begin{cor}\label{code bound}
Let $G = \prod_{i = 1}^t K_{n_i}$.  We have
$$\gamma(G) \leq 2^t - A(t,2,t-1) \leq 3\cdot 2^{t-2}.$$
\end{cor}
\begin{proof}
By Theorem \ref{GV bound}, we have $A(t,2) \geq 2^{t-1}$.  Let $M$ be a set of length $t$ binary vectors such that no two vectors are equal in $t-1$ coordinates and $|M| = 2^{t-1}$.  We can delete one vector in each pair of Boolean complements in $M$ to obtain a set $M' \subseteq M$ of size at least $2^{t-2}$ such that no two vectors are equal in $t-1$ coordinates nor differ in all coordinates.  Theorem \ref{dom code} implies $\{0,1\}^t \cut M'$ is a dominating set, hence
$$\gamma(G) \leq 2^t - 2^{t-2} = 3\cdot 2^{t-2}.\qedhere$$
\end{proof}

We now determine the domination number of a product of four complete graphs, extending  the results of Defant and Iyer in \cite{DI} and Meki{\v s} in \cite{Mek}.  Defant and Iyer determined $\gamma_t\left(\prod_{i = 1}^4 K_{n_i}\right)$ in the cases when $n_1 = 2$, $n_3 > 4 = n_1$, or $n_2 - 2 > 3 = n_1$.  Meki{\v s} determined $\gamma\left(\prod_{i = 1}^4 K_{n_i}\right)$ in the case that $n_1 \geq 5$.

\begin{thm}
If $2 \leq n_1  \leq n_2 \leq n_3 \leq n_4$ and $n_2 \geq 3$, then
$$\gamma\left(\prod_{i=1}^4 K_{n_i}\right) = \begin{cases} 8 & \text{ if } n_1 = 2 \\
7 & \text{ if } n_1 = 3, n_2 \leq 5\text{ or } n_1 = n_2 = n_3 = 4, n_4 \in\{4,5\}\\
6 & \text{ if } n_1 = 3, n_2 > 5 \text{ or } n_1 = 4, (n_3,n_4) \not\in \{(4,4),(4,5)\}\\
5 & \text{ if } n_1 \geq 5.
\end{cases}$$
Moreover, if $n_1 \geq 3$, then $\gamma\left( \prod_{i = 1}^4 K_{n_i}\right) = \gamma_t\left( \prod_{i = 1}^4 K_{n_i}\right)$.
\end{thm}
\begin{proof}
Let $G =  \prod_{i = 1}^4 K_{n_i}$.  Theorem 2.9 of \cite{DI} handles the case $n_1 = 2$.  The cases $n_3 > 4 = n_1$ or $n_2 - 2 > 3 = n_1$ follow from Theorem 2.8 of \cite{DI}.  The case $n_1 \geq 5$ is determined by Corollary 2.2 of \cite{Mek}.

We first show that if $n_1 \geq 3$, then $\gamma(G) \leq 7$.  Let our (total) dominating set be
$$D = \{(0,0,0,0),(1,0,1,1),(1,1,0,1),(1,0,1,1),(2,2,2,0),(2,2,0,2),(2,0,2,2)\}$$
Fix $(x_1,x_2,x_3,x_4) \in V(G)$.  Let $Y = \{(y_1,y_2,y_3,y_4) \in D : y_1 \not= x_1\}$.  We claim that $Y$ is a dominating set of $K_{n_2}\times K_{n_3} \times K_{n_4}$.  If $x_1 \in \{1,2\}$, then it is shown in \cite{Val} that $Y$ is dominating.  If $x_1 = 0$, observe that $Y$ consists of $6$ vertices where no three are equal in any given coordinate.  Moreover, no three disjoint pairs of vertices in $Y$ can be chosen such that each pair is equal in a distinct coordinate.  Thus $Y$ is dominating set of $K_{n_2}\times K_{n_3} \times K_{n_4}$ for any choice of $x_1$, so $D$ is a dominating set of $G$.

Next, we show that if $n_1 = 4$ and $n_3,n_4 \not\in \{(4,4),(4,5)\}$, then $\gamma(G) = 6$.  The lower bound follows from Theorem 2.6 of \cite{DI}.  For the upper bound, consider the set
$$D' = \{(0,0,0,0),(0,1,1,1),(1,0,1,2),(1,1,0,3),(4,4,4,4),(5,5,5,5)\}$$
Observe that there are no two pairs of vertices in $D'$ such that each pair is equal in a different coordinate.  Moreover, no three vertices of $D'$ are equal in any coordinate.  Therefore any vertex of $G$ is adjacent to a vertex of $D'$, hence $D'$ is a total dominating set of $G$.

It remains to show that $\gamma(G) \geq 7$ if $n_1 = 3, n_2 \leq 5\text{ or } n_1 = n_2 = n_3 = 4, n_4 \in\{4,5\}$. The lower bound in the case $n_1 = 3, n_2 \leq 5, n_3 \geq 5$ follows from Theorem 2.8 of \cite{DI}.  Thus we need only handle the case $n_3 \leq 4$.  Fix $\tilde D \subseteq (G)$ such that $|\tilde D| = 6$.  It is straightforward to show that if three vertices of $\tilde D$ are equal in some coordinate, then there are at least $6$ vertices which are not adjacent to a vertex of $\tilde D$, hence $\tilde D$ is not dominating. Observe that there must be at least seven tuples $(\{v,w\},i) \in \binom{\tilde D}{2} \times \{1,2,3,4\}$ such that the $i^{\text{th}}$ coordinates of $v$ and $w$ are equal.  It is straightforward to show that there exists two such tuples $(\{v,w\},i)$ and $(\{\hat v,\hat w\},\hat i)$ such that $\{v,w\} \cap \{\hat v, \hat w\} = \emptyset$ and $i \not= \hat i$.  Thus the vertex which is equal to $v$ in the $i^{\text{th}}$ coordinate, $\hat v$ in the $\hat i^{th}$ coordinate, and each vertex of $\tilde D \cut \{v,w,\hat v,\hat w\}$ in the remaining two coordinates is not dominated by $\tilde D$.  We can conclude that $\gamma(G) > 6$ if $n_3 \leq 4$.
\end{proof}

\subsection{Maximal Independent Sets in Products of Complete Graphs}
We begin by calculating the lower independence numbers of products of two or three complete graphs.

\begin{thm}\label{ind dom number}
Let $G = \prod_{i = 1}^t K_{n_i}$ for $2 \leq n_1 \leq \cdots \leq n_t$.  We have
$$i(G) = \begin{cases} n_1 & \text{ if } t = 2\\
4 &\text{ if } t = 3.\end{cases}$$
\end{thm}
\begin{proof}
Suppose $t = 2$.  If $n_1 = 2$, observe that $\{(0,0),(1,0)\}$ is a maximal independent set and $i(G) \geq \gamma(G) \geq 2$, hence $i(G) = n_1= 2$.  For $n_1 \geq 3$, we will show that the maximal independent sets are precisely the fibers under projection onto some coordinate.  Since the minimum size of such fibers is $n_1$, occurring when the projection is onto the second coordinate, this is sufficient.

Fix three independent vertices $(x_1,x_2),(y_1,y_2),(z_1,z_2) \in V(G)$.  It is impossible to have both $x_1 = y_1 \not= z_1$ and $y_2 = z_2 \not= x_2$, as in this case $x_1 \not= z_1$ and $x_2 \not= z_2$.  Thus, every set of three independent vertices must be equal in some coordinate.  We can conclude that the maximal independent sets must all be equal in some coordinate, hence can be extended to a fiber under the projection onto that coordinate.  Therefore $i(G) = n_1$.

Now we handle the case $t = 3$.  Observe that the set $\{(0,0,0),(1,0,1),(1,1,0),(0,1,1)\}$ forms a maximal independent set.  Theorem \ref{dom numbers} provides the lower bound $4 = \gamma(G) \leq i(G)$, so we can conclude $i(G) = 4$.
\end{proof}

\begin{cor}
For all integers $n_1,n_2,n_3\geq 3$, we have
\begin{align*}
    &\ir\left(K_2 \times K_{n_1}\right) = \alpha\left(K_2 \times K_{n_1}\right) = i\left(K_2 \times K_{n_1}\right),\\
    &\ir\left(K_3 \times K_{n_1}\right) = \alpha\left(K_3 \times K_{n_1}\right) = i\left(K_3 \times K_{n_1}\right),\\
    &\ir\left(K_{n_1} \times K_{n_2} \times K_{n_3}\right) = \alpha\left(K_{n_1} \times K_{n_2} \times K_{n_3}\right) = i\left(K_{n_1} \times K_{n_2} \times K_{n_3}\right).
\end{align*}
\end{cor}

We also briefly note an error in a paper by Uma Maheswari and Maheswari \cite{MahMah}.  Their paper correctly shows that $i(X_n) \leq \frac{n}{p_t}$, where $p_t$ is the largest prime divisor of $n$, but they incorrectly claim that equality holds for all $n$.  We show that the gap between this upper bound and the true value of $i(X_n)$ can be arbitrarily large.  

\begin{thm}
Let $p_1,\ldots,p_t$ be distinct primes, where $t\geq 3$. For $n=p_1\cdots p_t$, we have $i(X_n) \leq 4p_1p_2\cdots p_{t-3}$.
\end{thm}
\begin{proof}
Recall from the proof of Theorem \ref{ind dom number} that $E = \{(0,0,0),(1,0,1),(1,1,0),(0,1,1)\}$ is an independent dominating set of $K_{p_{t-2}} \times K_{p_{t-1}} \times K_{p_t}$.  We claim that the set 
$$D = \left(\prod_{i = 1}^{t-3} \{0,\dots,p_i - 1\} \right) \times \{(0,0,0),(1,0,1),(1,1,0),(0,1,1)\}$$
is an independent dominating set of $\prod_{i=1}^t K_{p_i} = X_n$.  For any vertex $x$ equal to an element of $D$ in the last three coordinates, we have $x \in D$.  For any vertex $x$ not equal to an element of $E$ in the last three coordinates, we can form a vertex $y \in D$ adjacent to $x$ by taking $y_i = x_i + 1$ in the first $t-3$ coordinates and choosing a dominating element of $E$ in the last three.  Lastly, we note that $D$ is independent as any two vertices of $D$ will be equal in at least one of the last three coordinates.
\end{proof}

Therefore, the upper bound given by Uma Maheswari and Maheswari \cite{MahMah} is tight only when $n$ has at most $2$ prime divisors.

\begin{cor}
For any $t \geq 3$ and $\epsilon > 0$, there exists a positive integer $n$ such that $\omega(n) = t$ and $\frac{i(X_n)}{n/p_t} < \epsilon$, where $p_t$ is the largest prime divisor of $n$.
\end{cor}

\section{Upper Domination Parameters of Products of Complete Multipartite Graphs}\label{Upp Dom}

We now shift our focus to the upper portion of the domination chain and broaden our scope to products of balanced, complete multipartite graphs.  In \cite{DI}, Defant and Iyer initiated the investigation of the upper domination parameters of products of balanced, complete multipartite graphs, proving the following result.

\begin{thm}\emph{(\cite{DI})}\label{DI upper}
Let $G = \prod_{i = 1}^t K[a_i,b_i]$ with $2 \leq b_1 \leq \cdots \leq b_t$.  If $b_1 = 2$ or $t \leq 3$, then $\alpha(G) = \IR(G)$.
\end{thm}

In fact, the original result stated only that $\alpha(G) = \Gamma(G)$, but their methods never use the hypothesis that $D$ is dominating, only that $D$ is irredundant.  In this section, we provide upper bounds for $\IR(G)$ in products of balanced, complete multipartite graphs not covered by Theorem \ref{DI upper}.

Fix an irredundant set $S \subseteq V(G)$. If $v \in S$ is isolated in $G[S]$, we say $v$ is {\it lonely}; otherwise, we say $v$ is {\it social}.  Let $\Lon(S)$ denote the set of lonely vertices in $S$, and let $\Soc(S)$ denote the set of social vertices in  $S$.  Observe that if $v \in S$ is social, then $\pn[v;S]\subseteq V(G) \cut S$.

As we did in the case of products of complete graphs, it is often useful to associate a $t$-tuple of integers to each vertex in $\prod_{i = 1}^t K[a_i,b_i]$.  Label the vertices of $K[a_i,b_i]$ by elements of $\Z/a_ib_i\Z$, where two vertices are adjacent if they are not congruent modulo $b_i$.  Associate to $v \in \prod_{i = 1}^t K[a_i,b_i]$ a vector of integers $(v(1),\dots,v(t))$, where the $v(i)$ is the element of $\{0,\dots,a_ib_i - 1\}$ associated to the projection of $v$ onto the $i^{th}$ coordinate.  Let $p(v)$ be the vector $(v(1) \mod b_1, v(2) \mod b_2, \dots, v(t) \mod b_t)$.  Note that vertices $u$ and $v$ are adjacent in $\prod_{i = 1}^t K[a_i,b_i]$ if and only if $p(u)$ differs from $p(v)$ in every coordinate.

\begin{lem}\label{same last coord}
Let $G = \prod_{i = 1}^t K[a_i,b_i]$, where $2 \leq b_1 \leq \cdots \leq b_t$. Let $S$ be an irredundant set in $G$, and let $T \subseteq \pn[S]$ be chosen so that each vertex of $S$ has exactly one private neighbor in $T$.  If there exist three vertices $v_1,v_2,v_3 \in T$ such that the vectors $p(v_1),p(v_2),p(v_3)$ vary in exactly one coordinate, then $v_1$, $v_2$, and $v_3$ are lonely vertices of $S$.
\end{lem}
\begin{proof}
\vspace{-0.1cm}
Without loss of generality, assume $p(v_1) = \{0\} \times \{u\}$, $p(v_2) = \{1\} \times \{u\}$ and $p(v_3) = \{2\} \times \{u\}$ for some vector $u$ of length $t-1$.  Observe that any vertex $w$ adjacent to at least one of $v_1,v_2,v_3$ is adjacent to at least two of $v_1,v_2,v_3$.  Therefore, by our assumption that each vertex of $S$ has exactly one private neighbor in $T$, we cannot have all three of $v_1,v_2,v_3 \in \pn[S] \cut \Lon[S]$.

Suppose exactly two of the vertices $v_1,v_2,v_3$ are not lonely; without loss of generality assume $v_1,v_2 \in \pn[S] \cut \Lon[S]$ and $v_3 \in \Lon[S]$.  Let $u_1$ be the vertex of $S$ adjacent to $v_1$ and let $u_2$ be the vertex of $S$ adjacent to $v_2$.  Since $u_1$ is not adjacent to $v_2$ and $u_2$ is not adjacent to $v_1$, we have $p(u_1) = \{1\} \times \{u'\}$ and $p(u_2) = \{0\} \times \{u''\}$ for some $(t-1)$-tuples $u',u''$ differing from $u$ in every coordinate.  However, $u_1$ and $u_2$ are both adjacent to $v_3$, contradicting that $v_3$ is lonely.

Now, suppose exactly one of the vertices $v_1,v_2,v_3$  is not a lonely vertex of $S$; without loss of generality assume $v_1 \in \pn[S]\cut\Lon[S]$ and $v_2,v_3 \in \Lon[S]$. Note that there is a vertex $w \in S$ such that $p(w) \in \{1,2,\dots,b_1 - 1\} \times \{u'\}$ for some $(t-1)$-tuple $u'$ that differs from $u$ in every coordinate.  However, at least one of $v_2$ or $v_3$ is adjacent to $w$, contradicting that $v_2$ and $v_3$ are lonely.  Therefore, $v_1,v_2,v_3 \in \Lon[S]$.
\end{proof}

\begin{lem}\label{same assoc vector}
Let $G = \prod_{i = 1}^t K[a_i,b_i]$, where $2 \leq b_1 \leq \cdots \leq b_t$.  Let $S$ be an irredundant set in $G$, and let $T \subseteq \pn[S]$ be chosen so that each vertex of $S$ has exactly one private neighbor in $T$. If there are two vertices $v_1,v_2 \in T$ such that $p(v_1) = p(v_2)$, then $v_1$ and $v_2$ are lonely vertices of $S$. 
\end{lem}
\begin{proof}
Note that $N(v_1) = N(v_2)$.  Suppose, seeking a contradiction, that $v_1 \in T \cut \Lon[S]$.  Then there exists some $w \in S$ such that $w$ is adjacent to $v_1$ and $v_2$, so $v_2$ cannot be lonely.  Hence $v_2$ must also be a private neighbor of some vertex in $S$, but as is it already a neighbor of $w \in S$, this contradicts our assumption that each vertex in $S$ has a unique private neighbor in $T$.
\end{proof}

\begin{thm}\label{up dom bd 1}
If $G = \prod_{i = 1}^t K[a_i,b_i]$, where $2 \leq b_1 \leq \cdots \leq b_t$, then 
$$\IR(G) \leq \frac{1}{b_1}\prod_{i=1}^t a_ib_i + \frac{2}{b_t} \prod_{i = 1}^t b_i.$$
\end{thm}
\begin{proof}
Suppose $S \subseteq V(G)$ is irredundant. Observe that $|\Lon(S)|\leq \alpha(G)=\frac{1}{b_1}\prod_{i=1}^t a_ib_i$.  Let $T \subseteq \pn[S]$ be chosen so that each vertex of $S$ has exactly one private neighbor in $T$.  By Lemma \ref{same assoc vector}, no two vertices of $T \cut \Lon[S]$ have the same associated vector.  Moreover, by Lemma \ref{same last coord}, no three vertices of $T \cut \Lon[S]$ have associated vectors which agree in all but the last coordinate.  Hence 
$$|\Soc(S)|  = |T \cut \Lon[S]| \leq \frac{2}{b_t}\prod_{i = 1}^t b_i.$$
We can conclude
$$|S| = |\Lon(S)| + |\Soc(S)| \leq \frac{1}{b_1}\prod_{i=1}^t a_ib_i + \frac{2}{b_t}\prod_{i = 1}^t b_i.\qedhere$$
\end{proof}
In particular, for the case of unitary Cayley graphs of $\Z/n\Z$, we have the following bound.

\begin{cor}
If $n = p_1^{\alpha_1}p_2^{\alpha_2}\cdots p_t^{\alpha_t}$, where $p_1 < \cdots < p_t$ are primes, then
$$\IR(X_n) \leq \left(1 + 2 \cdot \frac{p_1}{p_t}\cdot \frac{1}{p_1^{\alpha_1-1}p_2^{\alpha_2-1}\cdots p_t^{\alpha_t - 1}}\right)\alpha(X_n).$$
\end{cor}

If we look at sets of integers with a bounded smallest prime factor, we can show that $\IR(X_n)$ and $\alpha(X_n)$ are asymptotically the same. An integer $x$ is called {\it $r$-rough} (or {\it $r$-jagged}) for a positive integer $r$ if every prime factor of $x$ is at least $r$. Let $D_r = \{x \in \N: \exists d \in \{2,3,\dots,r\} \text{ such that } d \mid x\}$; that is, $D_r$ is the set of positive integers that are not $(r+1)$-rough.

\begin{cor}
For each $r \in \N$, we have
$$\lim_{\substack{n \in D_r \\ n \to \infty}} \frac{\IR(X_n)}{\alpha(X_n)} = 1.$$
\end{cor}
\begin{proof}
Let $n = p_1^{\alpha_1}\cdots p_n^{\alpha_n}$, where $p_1 < \cdots < p_n$ are primes.  Hence $X_n \cong \prod_{i = 1}^t K[a_i,b_i]$, where $a_i = p^{\alpha_i - 1}$ and $b_i = p_i$.  By the condition $n \in D_r$, we know $b_1 \leq r$.  Hence 
$$\frac{\IR(X_n)}{\alpha(X_n)}  \leq 1 + \frac{2b_1}{b_t\prod_{i=1}^t a_i} \leq 1 + \frac{2r}{b_t\prod_{i=1}^t a_i}.$$
Note that $b_t \prod_{i = 1}^t a_i$ tends to infinity as $n$ tends to infinity.  Since $\IR(X_n) \geq \alpha(X_n)$ for all $n$, the limit goes to $1$.
\end{proof}

We note that the bound in Theorem \ref{up dom bd 1} is trivial if $a_1 = \cdots = a_t = 1$ and  $b_1 =\cdots =  b_t = 3$, that is, when $G$ is a direct product of triangles.  Defant and Iyer \cite{DI} drew attention to direct products of triangles as a ``particularly attractive" special case of their conjecture concerning the value of the upper domination number for direct products of balanced, multipartite graphs.

\begin{conj}\emph{(\cite{DI})}
If $G = \prod_{i = 1}^t K[a_i,b_i]$ with $2 \leq b_1 \leq \cdots \leq b_t$, then $\Gamma(G) = \alpha(G)$.
\end{conj}
This case motivates the following theorem, which provides a nontrivial (though, usually worse) upper bound for the upper irredundance number of any direct product of balanced, complete multipartite graphs.

\begin{thm}
If $G = \prod_{i = 1}^t K[a_i,b_i]$, where $2 \leq b_1 \leq \cdots \leq b_t$, then
$$\IR(G) \leq \frac{b_1}{2b_1 - 1}\prod_{i=1}^t a_ib_i.$$
\end{thm}
\begin{proof}
Let $S \subseteq V(G)$ have size greater than $\frac{1}{2 - \frac{1}{b_1}} \prod_{i=1}^t a_ib_i$. Suppose, seeking a contradiction, that $S$ is irredundant.

Let $P_j$ be the set of vertices $v \in S$ such that $v(1) = j$. By the Pigeonhole Principle, $|P_j| \geq \frac{1}{b_1}|S|$ for some $j \in \{0,\dots,b_1 -1\}$; without loss of generality, suppose $|P_0| \geq \frac{1}{b_1}|S|$.

For each $k \in \{1,\dots,b_1 - 1\}$, let $Q_k$ denote the set of private neighbors $u \in \bigcup_{v \in (S \cut P_0)} \pn[v,S]$ such that $u(1)= k$.  Observe that of the $|S| - |P_0|$ vertices in $S \cut P_0$, at most $\frac{1}{b_1}\prod_{i = 1}^t a_ib_i - |P_0|$ have private neighbors $u$ such that $u(1) = 0$.  By the Pigeonhole Principle, 
$$|Q_k| \geq \frac{1}{b_1 - 1}\left(|S| - |P_0| - |\{u \in \pn[S \cut P_0] : u(1) = 0\}|\right) \geq \frac{1}{b_1 - 1}\left(|S| - \frac{1}{b_1}\prod_{i = 1}^t a_ib_i\right)$$
for some $k \in \{1,\dots,b_1 - 1\}$; without loss of generality, assume $|Q_1| \geq \frac{1}{b_1 - 1}\left(|S| - \frac{1}{b_1}\prod_{i = 1}^t a_ib_i\right)$.  Let $R$ be the set of vertices $u$ such that $p(u) = (0,v(2)+ 1,\dots,v(t) + 1)$ for some $v \in Q_1$.  That is,
$$R = \bigcup_{v \in Q_1} p^{-1}((0,v(2) + 1,\dots,v(t) + 1)).$$
\begin{figure}[h]
    \centering
    \includegraphics[scale=0.4]{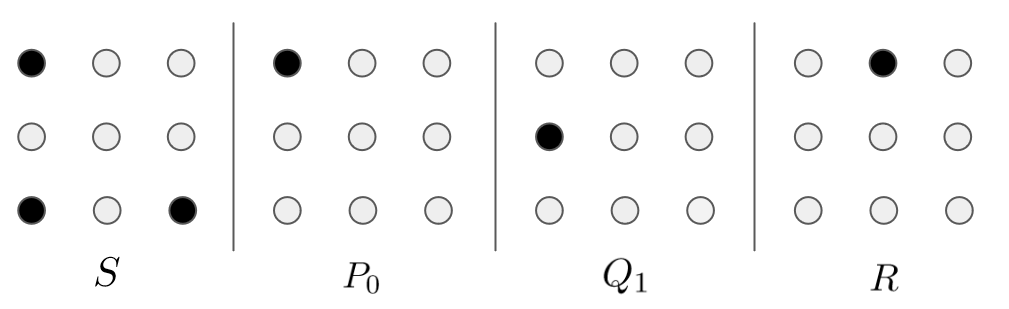}
    \caption{An example of how the sets $P_0$, $Q_1$, and $R$ are determined by a fixed irredundant set $S = \{(0,0),(2,0),(2,2)\}$ in $K_3 \times K_3$. For a vertex $v$ in the figure, $v(0)$ is given by its row and $v(1)$ by its column, hence two vertices are adjacent if and only if they lie in different rows and columns.   }
    \label{fib}
\end{figure}
Since the associated vectors are disjoint, each member of $R$ is adjacent to a member of $Q_1$.  Using the fact that $P_0$ is disjoint from $Q_1$ and that $Q_1$ consists of private neighbors of $S \cut P_0$, we conclude that $P_0$ and $R$ are disjoint; lest a member of $P_0$ is adjacent to a private neighbor of $S \cut P_0$.  Since $P_0$ and $R$ both consist of vectors $u$ satisfying $u_1 = 0$, we have
$$|P_0| + |R| \leq \frac{1}{b_1} \prod_{i = 1}^t a_ib_i.$$
However, we reach a contradiction by noting that
\begin{align*}
    |P_0| + |R| &= |P_0| + |Q_1|\\
    &\geq \frac{1}{b_1}|S| + \frac{1}{b_1 - 1}\left(|S| - \frac{1}{b_1}\prod_{i = 1}^t a_ib_i\right)\\
    &> \left( \frac{1}{2b_1 - 1}+  \frac{1}{b_1 - 1}\left(\frac{b_1}{2b_1 - 1} - \frac{1}{b_1}\right)\right) \prod_{i=1}^t a_ib_i\\
    &= \frac{1}{b_1} \prod_{i = 1}^t a_ib_i.
\end{align*}
We conclude that $S$ is not irredundant.
\end{proof}

\section{Further Directions}

In Section \ref{Dom Sets}, we raise the question of whether there exist infinitely many integers $n$ such that $\gamma_c(X_n) > g(n)$ and, if so, whether such integers can have arbitrarily many distinct prime factors.  At the end of this same section, we ask whether there exists a single integer $n$ such that $\gamma_t(X_n) \leq g(n) - 3$.  

We calculate the irredundance and lower independence numbers of direct products of at most three complete graphs in Section \ref{Low Dom}.  It remains open to determine these parameters for larger products of complete graphs.  From these calculations, it follows that $\ir(X_n) = i(X_n)$ when $n$ is prime, $n = 2p$, or $n = 3p$ for some prime $p$, or $n$ is squarefree with exactly three prime divisors.  We pose the problem of finding other squarefree integers $n$ for which equality is achieved in the lower portion of the domination chain.

We note that the irredundance, domination, and lower independence numbers of $K_a \times K_b$ depend on $\min(a,b)$, whereas in the case of $K_a \times K_b \times K_c$ these parameters are independent of $a$, $b$, and $c$.  We show that the domination number of a direct product of four complete graphs is dependent on the size of all four graphs in the product.  We pose the question of determining for which integers $t$ do the irredundance, domination, or lower independence numbers of $\prod_{i=1}^t K_{n_i}$ depend on all of $n_1,n_2,\dots,n_t$, where $n_1 \leq n_2 \leq \cdots \leq n_t$.

As discussed in Section \ref{Upp Dom}, the methods of Defant and Iyer in \cite{DI} for calculating the upper domination number of products of balanced, complete multipartite graphs are easily adapted for calculating the upper irredundance number of these graphs. This suggests a strengthening of their conjecture.

\begin{conj}
If $G = \prod_{i=1}^t K[a_i,b_i]$ with $2 \leq b_1 \leq \cdots \leq b_t$, then $\IR(G) = \alpha(G)$.
\end{conj}

In 2007, Klotz and Sander \cite{KS} introduced the notion of the {\it gcd-graph} $X_n(D)$, the graph on $\{0,\dots,n-1\}$ where vertices $x$ and $y$ are adjacent if and only if $\gcd(|x-y|,n) \in D \subseteq \N$.  In particular, Klotz and Sander show that all eigenvalues of $X_n(D)$ are integral.  It may be interesting to investigate domination parameters in the more general case of gcd-graphs.

\section{Acknowledgements}
The author would like to thank Joe Gallian for his tireless efforts to foster a productive and engaging mathematical community.  She extends her gratitude to Colin Defant for introducing her to this problem and providing incredible support throughout her research.  This research was conducted at the University of Minnesota, Duluth  REU  and  was  supported  by  NSF/DMS  grant  1650947  and  NSA grant H98230-18-1-0010.

\end{document}